\documentclass[letterpaper, 10 pt, conference]{ieeeconf}  

\IEEEoverridecommandlockouts                              

\usepackage{amsmath, amssymb, bbm, xspace}

\newcommand{\dom}{\mathrm{dom}}

\def\spose#1{\hbox to 0pt{#1\hss}}

\def\text #1{\hbox{\quad#1\quad}}

\def\nthinsp{\mskip -2   mu}

\def\A{_{\scriptscriptstyle A}}

\def\F{_{\scriptscriptstyle F}}

\def\R{_{\scriptscriptstyle R}}

\def\superstar{^{\raise 0.5pt\hbox{$\nthinsp *$}}}
\def\SUPERSTAR{^{\raise 0.5pt\hbox{$*$}}}

\def\lamstarT {\lambda^{\raise 0.5pt\hbox{$\nthinsp *$}T}}

\def\hbar{\skew{4.2}\bar h}

		\def\bkE{{\rm I\kern-.17em E}}
		\def\bk1{{\rm 1\kern-.17em l}}
		\def\bkD{{\rm I\kern-.17em D}}
		\def\bkR{{\rm I\kern-.17em R}}
		\def\bkP{{\rm I\kern-.17em P}}
		\def\bkY{{\bf \kern-.17em Y}}
		\def\bkZ{{\bf \kern-.17em Z}}

		\def\beq{\begin{eqnarray}}
		\def\bc{\begin{center}}
		\def\be{\begin{enumerate}}
		\def\bi{\begin{itemize}}
		\def\bS{\begin{slide}}
		\def\ec{\end{center}}
		\def\ee{\end{enumerate}}
		\def\ei{\end{itemize}}
		\def\eS{\end{slide}}
		\def\eeq{\end{eqnarray}}

	\def\cp2problem#1#2#3#4{\fbox
		 {\begin{tabular*}{0.9\textwidth}
			{@{}l@{\extracolsep{\fill}}l@{\extracolsep{6pt}}l@{\extracolsep{\fill}}c@{}}
				#1 & & $#4 $ 
			\end{tabular*}}}

		\renewcommand{\emph}[1]{\textbf{#1}}

		\def\bkE{{\rm I\kern-.17em E}}
		\def\bk1{{\rm 1\kern-.17em l}}
		\def\bkD{{\rm I\kern-.17em D}}
		\def\bkR{{\rm I\kern-.17em R}}
		\def\bkP{{\rm I\kern-.17em P}}
		
		\def\bkZ{{\bf{Z}}}

\newcommand {\beeq}[1]{\begin{equation}\label{#1}}
\newcommand {\eeeq}{\end{equation}}
\newcommand {\bea}{\begin{eqnarray}}
\newcommand {\eea}{\end{eqnarray}}

\def\texitem#1{\par\smallskip\noindent\hangindent 25pt
               \hbox to 25pt {\hss #1 ~}\ignorespaces}

\newtheorem{definition}{Definition}{\it}{}
{\it}{}
{\it}{}
{\it}{}
\newtheorem{lemma}{Lemma}{\it}{}
\newtheorem{theorem}{Theorem}{\it}{}
\newtheorem{remark}{Remark}{\it}{}
{\it}{}
\newtheorem{standing}{Standing Assumption}

\def\R{\mathbb{R}}
\def\N{\mathbb{N}}

\def\F{\mathcal{F}}

\def\argmin{\mathop{\rm argmin}}

\def\x{\bs{x}}
\def\y{\bs{y}}
\def\J{\mathrm{J}{}}
\def\zer{\operatorname{zer}}

\def\nc{\mathrm{N}}
\def\A{\mc{A}}

\def\ber{\begin{eqnarray}}
\def\eer{\end{eqnarray}}
\def\bers{\begin{eqnarray*}}
	\def\eers{\end{eqnarray*}}
\def\be{\begin{equation}}
\def\ee{\end{equation}}
\def\1{{\bf 1}}

\newcommand{\bs}{\boldsymbol}
\newcommand{\mc}{\mathcal}

\newcommand{\col}{\mathrm{col}}

\renewcommand{\emph}{\textit}
\def\fix{\mathrm{fix}}
\newcommand{\0}{\bs 0}
\renewcommand{\iff}{\; \Leftrightarrow \;}

\newcommand{\Id}{\mathrm{Id}}
\newcommand{\alphamax}{\alpha_{\textnormal{max}}}

\newcommand{\gra}{\operatorname{gra}}

\usepackage{amsmath} 
\usepackage{amsfonts,mathrsfs}
\usepackage{amssymb}
\usepackage{mathtools}
\usepackage{color}
\usepackage{graphicx}
\usepackage{epsfig}
\usepackage{algorithm}
\usepackage{algorithmicx}
\usepackage[acronym]{glossaries}
\usepackage{subfigure}
\usepackage{cancel}
\usepackage{empheq}
\usepackage{placeins}

\let\labelindent\relax    
\usepackage{enumitem}

\makeglossaries
\newacronym{GNEP}{GNEP}{generalized Nash equilibrium problem}
\newacronym{NE}{NE}{Nash equilibrium}
\newacronym{NEP}{NEP}{Nash equilibrium problem}
\newacronym{GNE}{GNE}{generalized Nash equilibrium}
\newacronym{v-GNE}{v-GNE}{variational \gls{GNE}}
\newacronym{ISS}{ISS}{Input-to-state-stable}
\newacronym{PPPA}{PPPA}{preconditioned proximal-point algorithm}
\newacronym{PPA}{PPA}{proximal-point algorithm}
\newacronym{VI}{VI}{variational inequality}
\newacronym{GAE}{GAE}{generalized aggregative equilibrium}
\newacronym{v-GAE}{v-GAE}{variational \gls{GAE}}
\newacronym{KKT}{KKT}{Karush-Kuhn-Tucker}
\newacronym{FQNE}{FQNE}{firmly quasinonexpansive}
\newacronym{FNE}{FNE}{firmly nonexpansive}
\DeclareSymbolFont{myletters}{OML}{ztmcm}{m}{it}
\DeclareMathSymbol{\uplambda}{\mathord}{myletters}{"15}
\def\QEDhereeqn{\eqno\let\eqno\relax\let\leqno\relax\let\veqno\relax\hbox{\QED}}
\def\QEDopenhereeqn{\eqno\let\eqno\relax\let\leqno\relax\let\veqno\relax\hbox{\QEDopen}}

\usepackage{empheq}
\usepackage{cancel}

 \usepackage{cite}

\title{\LARGE \bf
A fully-distributed proximal-point algorithm for Nash equilibrium seeking  with linear convergence rate
}

\author{Mattia Bianchi, Giuseppe Belgioioso, and Sergio Grammatico
\thanks{\noindent M. Bianchi and S. Grammatico are with the Delft Center for Systems and Control (DCSC), TU Delft, The Netherlands. G. Belgioioso is with the Automatic Control Laboratory, Swiss Federal Institute of Technology (ETH) Z\"{u}rich, Switzerland. E-mail addresses: \texttt{\{m.bianchi, s.grammatico\}@tudelft.nl}, \texttt{gbelgioioso@ethz.ch}. This work was partially supported by NWO  (project OMEGA,  613.001.702) and by the ERC (project COSMOS, 802348).} 
}

\begin{document}

\maketitle
\thispagestyle{empty}
\pagestyle{empty}

\begin{abstract}
	We address the Nash equilibrium  problem in a  partial-decision information scenario, where each agent can only observe the actions of some neighbors, while its cost possibly depends on the strategies of other agents. Our main contribution is the  design of a fully-distributed, single-layer, fixed-step  algorithm, based on a proximal best-response augmented with consensus terms. To derive our algorithm, we follow an operator-theoretic approach. 
	First, we recast the Nash equilibrium problem as that of finding a zero of a monotone operator. Then, we demonstrate that the resulting inclusion can be solved in a fully-distributed way via a proximal-point method, thanks to the use of a novel preconditioning matrix.
Under strong monotonicity and Lipschitz continuity of the game mapping, we prove linear convergence of our algorithm to a Nash equilibrium.
Furthermore, we show that our method outperforms the fastest  known gradient-based schemes, both in terms of guaranteed convergence rate, via theoretical analysis, and in practice, via numerical simulations. 
%
%
\end{abstract}

\section{Introduction}\label{sec:introduction}
\gls{NE} problems have received increasing attention with the spreading of networked systems, due to the numerous engineering applications, including communication networks \cite{Palomar_Eldar_Facchinei_Pang_2009}, 
formation control \cite{Stankovic_Johansson_Stipanovic_2012}, charge scheduling of electric vehicles \cite{Grammatico2017} and  demand response in competitive markets \cite{Li_Chen_Dahleh_2015}. 
These scenarios are characterized by the presence of   multiple selfish decision-makers, or agents, that aim at optimizing their individual, yet inter-dependent, objective functions. 
From a game-theoretic perspective, one of the challenges 
is to assign to the agents behavioral rules that eventually ensure the attainment of a \gls{NE}, a joint action from which no agent has interest to unilaterally deviate.

\emph{Literature review:}
Typically, \gls{NE} seeking algorithms are designed under the assumption that each agent can access the decisions of all the competitors, for example  in the presence of a coordinator that broadcasts the data to the network \cite{Yu_VanderSchaar_Sayed_2017,BelgioiosoGrammatico_ECC_2018,Shamma_Arslan_2005,YiPavel2019}. 
 However the existence of a central node with bidirectional communication with all the agents is impractical for many applications  \cite{Ghaderi_2014,Bimpikis2014}. One example is the Nash-Cournot  competition model described in \cite{Koshal_Nedic_Shanbag_2016}, where the  profit of each of a group of firms depends not only on its own production, but also on the whole amount of sales, a quantity not directly accessible by any of the firms.
This motivates the development of fully-distributed algorithms, which allow to compute \glspl{NE} relying on local data  only. 
Two main approaches have been proposed, corresponding to  two different information structures. For games where each agent can measure its own cost functions, pay-off based schemes were developed that do not require peer-to-peer communication  \cite{Stankovic_Johansson_Stipanovic_2012}, \cite{FrihaufKrsticBasar_2012}. Instead, we consider the so-called \emph{partial-decision information} scenario, where the  agents hold an analytic expression of their own cost functions, but  they are unable to evaluate the actual values, since they cannot access the strategies of all the competitors. To remedy the lack of knowledge, the agents
engage in nonstrategic information exchange
with some neighbors on a network;  from the  data received, they can estimate the strategies of all other agents, and eventually reconstruct the true values. 
This setup has only been introduced very recently. Most of the results available  resort to (projected) gradient 
and consensus dynamics,
 both in continuous time \cite{DePersisGrammatico2018,GadjovPavel2018}, and discrete time.
For the discrete time case, 
early works \cite{Koshal_Nedic_Shanbag_2016},  \cite{Salehisadaghiani_Pavel_GOSSIP}, focused on algorithms with vanishing step sizes, which typically result in slow convergence. 
More recently, fixed-step schemes were introduced  in \cite{SalehisadaghianiWeiPavel2019,TatarenkoShiNedic2018,Pavel2018}, building on a restricted monotonicity property, first revealed in \cite{GadjovPavel2018}.
The drawback is that, due to the partial-decision information assumption, small step sizes have to be chosen,  affecting the speed of convergence.
Of particular interest for this paper is the technique developed by Pavel in \cite{Pavel2018}, that characterized the equilibria of a (generalized) game as the zeros of a monotone operator.  The operator-theoretic approach is very elegant and convenient, since several splittings methods are already well established to solve monotone inclusions \cite[§26]{Bauschke2017}.  For example, the authors of  \cite{YiPavel_PPP_2019} adopted a \gls{PPPA}; yet, this results in a double-layer scheme, where the agents have to communicate multiple times to solve (inexactly) a subgame, at each step. Similarly, the proximal best-response dynamics proposed in \cite{Lei_Shanbag_CDC2018} for stochastic games require an increasing number of data transmissions per iteration. 

\emph{Contribution:} Motivated by the above, in this paper we further exploit the restricted monotonicity property  used in \cite{SalehisadaghianiWeiPavel2019,TatarenkoShiNedic2018,Pavel2018} to solve Nash equilibrium problems under partial-decision information. Specifically:
\begin{itemize}[leftmargin=*,topsep=0.5pt]
	\item We derive a simple, fully-distributed \gls{PPA}, that is a proximal best-response augmented with consensual terms. Thanks to the use of a novel preconditioning matrix, our algorithm is single-layer, i.e., it requires only one communication per iteration.  To the best of our knowledge, our \gls{PPPA} is the first non-gradient based algorithm with this feature (§\ref{sec:distributedGNE}-§\ref{sec:convergence});
	\item
	We prove global linear convergence of our algorithm to a NE, under strong monotonicity and Lipschitz continuity of the game mapping,
	by providing a general result for the  \gls{PPA} of restricted strongly monotone operators  (§\ref{sec:convergence});
	\item We demonstrate,  both by comparing the theoretical convergence rates and in simulation, that our algorithm outperforms existing gradient-based dynamics, in terms of the number of iterations needed to converge (§\ref{sec:theoreticalrates}-§\ref{sec:numerics}).
\end{itemize}
%
%

\smallskip
\emph{Basic notation}: 
See  \cite{Bianchi_LCSS2020}.
\newline\indent
\emph{Operator-theoretic notation}: 
For a function $\psi: \R^n \rightarrow \R \cup \{\infty\}$, $\dom(\psi) = \{x \in \R^n \mid \psi(x) < \infty\}$.  The mapping $\iota_{S}:\R^n \rightarrow \{ 0, \, \infty \}$ denotes the indicator function for the set $\mc{S} \subseteq \R^n$, i.e., $\iota_{S}(x) = 0$ if $x \in S$, $\infty$ otherwise. 
A set-valued mapping (or operator) $\mc{F}:\R^n\rightrightarrows \R^n$ is characterized by its graph
$\gra (\mc{F})=\{(x,u) \mid u\in \mc{F}(x)\}$. $\dom(\mc{F})=\{x\in\R^n| \mc{F}(x)\neq \varnothing \}$,
$\fix\left( \mathcal{F}\right) = \left\{ x \in \R^n \mid x \in \mathcal{F}(x) \right\}$ and $\zer\left( \mathcal{F}\right) = \left\{ x \in \R^n \mid 0 \in \mathcal{F}(x) \right\}$ denote the domain, set of fixed points and set of zeros, respectively. $\mc{F}^{-1} $ denotes the inverse operator of $\mc{F}$, defined through its graph as $\gra (\mc{F}^{-1})=\{(u,x)\mid (x,u)\in \gra(\mc{F})\}$. $\mathcal{F} : \R^n \rightrightarrows \R^n$ is
($\mu$-strongly) monotone  if $\langle u-v,x-y\rangle \geq 0 \, (\geq \mu \|x-y\|^2 )$, for all $(x,u)$,$(y,v)\in\gra(\F)$, where $\langle \cdot, \cdot \rangle$ denotes the Euclidean inner product. 
$\Id$ denotes the identity operator. 
For a function $\psi: \R^n \rightarrow \R \cup \{\infty\}$,
$\partial \psi: \dom(\psi) \rightrightarrows \R^n$ denotes its subdifferential operator, defined as $\partial \psi(x) = \{ v \in \R^n \mid \psi(z) \geq \psi(x) + \langle v, z-x \rangle   \textup{ for all } z \in {\rm dom}(\psi) \}$. $\nc_{S}: \R^n \rightrightarrows \R^n$ is the normal cone operator for the the set $S \subseteq \R^n$, i.e., 
$\nc_{S}(x) = \varnothing$ if $x \notin S$, $\left\{ v \in \R^n \mid \sup_{z \in S} \, \langle v, z-x\rangle  \leq 0  \right\}$ otherwise. If $S$ is closed and convex, it holds that $\partial \iota_S=\nc_S$.
 ${\rm J}_{\mathcal{F} }:=(\Id + \mathcal{F} )^{-1}$ denotes the resolvent operator of $\mathcal{F} $. 

\section{Mathematical background}\label{sec:mathbackground}
We consider a set of agents  $ \mc I:=\{ 1,\ldots,N \}$. Each agent $i\in \mc{I}$ shall choose its decision variable (i.e., strategy) $x_i$ from its local decision set $\textstyle \Omega_i \subseteq \R^{n_i}$. Let $x = \col( (x_i)_{i \in \mc I})  \in \Omega $ denote the vector of all the agents' decisions, $\textstyle \Omega = \Omega_1\times\dots\Omega_N\subseteq \R^n$ the overall action space and $\textstyle n:=\sum_{i=1}^N n_i$. 
The goal of each agent $i \in \mc I$ is to minimize its objective function $J_i(x_i,x_{-i})$, which depends on both the local variable $x_i$ and on the decision variables of the other agents $x_{-i}:= \col( (x_j)_{j\in \mc I\backslash \{ i \} } )$.
Then, the game is represented by the  inter-dependent optimization problems:
\begin{align} \label{eq:game}
\forall i \in \mc{I}:
\quad \underset{y_i \in \Omega_i}{\argmin}  \; J_i(y_i,x_{-i}).
	\end{align}
Our goal here is
to compute a \gls{NE}, as defined next.

\begin{definition}
	A Nash equilibrium is a set of strategies $x^{*}=\operatorname{col}\left((x_{i}^{*})_{i \in \mathcal{I}}\right)\in \Omega$ such that 
	\[
	\forall i \in \mathcal{I}: \quad x_{i}^{*} \in \underset{y_i\in \Omega_i}{\argmin} \, J_{i}\left(y_{i}, x_{-i}^{*}\right). \QEDopenhereeqn
	\]
\end{definition}

\medskip
The following assumptions are standard for \gls{NE} problems, see, e.g., \cite[Ass.~1]{TatarenkoShiNedic2018}, \cite[Ass.~2]{GadjovPavel2018}.

\begin{standing}[Regularity and convexity]\label{Ass:Convexity}
	For each $i\in \mathcal{I}$, the set $\Omega_i$ is non-empty, closed and convex; $J_{i}$ is continuous and the function $J_{i}\left(\cdot, x_{-i}\right)$ is continuously differentiable and convex for every $x_{-i}$.
	{\hfill $\square$} \end{standing}

Under Standing Assumption~\ref{Ass:Convexity}, a strategy $x^*$ is a \gls{NE} of the game in \eqref{eq:game} if and only if 
%
 it is a solution of the variational inequality VI$(F,\Omega)$\footnote{Given a set $S\subseteq \R^m$ and a mapping $\psi:S\rightarrow \R^m$, the variational inequality VI$(\psi,S)$ is the problem of finding a vector $\omega^*\in S$ such that $\psi(\omega^*)^\top(\omega-\omega^*)\geq 0$, for all $\omega \in S$.}  \cite[Prop.~1.4.2]{FacchineiPang2007},
where $F$ is the \emph{pseudo-gradient} mapping of the game:
\begin{align}
\label{eq:pseudo-gradient}
F(x):=\operatorname{col}\left( (\nabla _{x_i} J_i(x_i,x_{-i}))_{i\in\mathcal{I}}\right).
\end{align}
Equivalently, $x^*$ is a \gls{NE} if and only if the following holds:
\begin{align} \label{eq:NEinclusion}
{\0_{n}} & {\in F\left(x^{*}\right)+\mathrm{N}_{\Omega}\left(x^{*}\right)},
\end{align}
A sufficient condition for the existence of a unique \gls{NE} for the game in \eqref{eq:game} is the strong monotonicity of the pseudo-gradient \cite[Th. 2.3.3]{FacchineiPang2007}, as postulated next. This assumption is always used for \gls{NE} seeking under partial-decision information with fixed step sizes, e.g.,   \cite[Ass.~2]{Pavel2018}, \cite[Ass.~4]{DePersisGrammatico2018}, \cite[Ass.~2]{TatarenkoShiNedic2018}.

\begin{standing}\label{Ass:StrMon}
	The pseudo-gradient mapping $F$ in \eqref{eq:pseudo-gradient}  is $\mu$-strongly monotone and $\theta_0$-Lipschitz continuous, for some $\mu, \theta_0>0$: for any $x,y\in\R^n$, $\textstyle \langle x-y, F(x)-F(y)\rangle \geq \mu\|x-y\|^2$ and $\|F(x)-F(y)\|\leq \theta_0 \|x-y\|$. 
	\hfill $\square$
\end{standing}

\section{Distributed  Nash equilibrium seeking}\label{sec:distributedGNE}
 \setlength{\dbltextfloatsep}{12pt}
\begin{algorithm*}[t] \caption{Fully-distributed Nash equilibrium seeking via  preconditioned proximal-point iteration} \label{algo:1}
	\vspace{0.3em}
	{Initialization}: ~~~~~For all $i\in \mc{I}$, set $x_i^0\in \Omega_i$, $\bs{x}_{i,-i}^0\in \R^{n-n_i}$.
	\begin{itemize}[leftmargin=6.6em]
		\item[For all $k\geq 0$: ]
		
		\begin{itemize}[leftmargin=1em]
			\item[]  
			Communication: The  agents exchange the variables  $\{ x^k_{i},\bs{x}_{i,-i}^k\}$ with their neighbors.
			\\
			Local variables update: each agent  $i\in \mc{I}$ does:
			\begin{align*}
			\begin{aligned}
			\x_{i,-i}^{k+1}&= \textstyle \frac{1}{2}(\x_{i,-i}^k+\textstyle\sum_{j=1 }^{N}w_{i,j}\x_{j,-i}^k)
			\\
			x_i^{k+1} & =\begin{aligned}[t]
			\underset{y \in \Omega_i} {\argmin}\bigl(  &
			J_i(y,\x_{i,-i}^{k+1})+\textstyle \frac{1}{2\alpha}  \bigl\| y-x_i^{k} \bigr\|^2
			+\textstyle \frac{1}{2\alpha} 
			\bigl\| \textstyle y-\sum_{j=1}^{N}w_{i,j}\x_{j,i}^{k} \bigr\|^2  \bigr)
			\end{aligned}
			\end{aligned}
			\end{align*}
		\end{itemize}
	\end{itemize}
	\vspace{-0.5em}
\end{algorithm*}

In this section, we present an algorithm to seek a \gls{NE} of the game in \eqref{eq:game} in a fully-distributed way.
Specifically, 
each agent $i$ only knows its own cost function $J_i$ and feasible set $\Omega_i$. 
Moreover, agent $i$ does not have full knowledge of $x_{-i}$, and only relies on the information exchanged locally with some neighbors over a communication network $\mathcal G(\mc{I},\mc{E})$. The unordered pair $(i,j) $ belongs to the set of edges $\mc{E}$ if and only if agent $i$ and $j$ can mutually exchange information. 
We denote: $W\in \R^{N\times N}$ the  mixing matrix of $\mc{G}$, with $w_{i,j}:=[W]_{i,j}$ and  $w_{i,j}>0$ if $(i,j)\in \mc{E}$, $w_{i,j}=0$ otherwise;
$\mc{N}_i=\{j\mid (i,j)\in \mc{E}\}$ the set of neighbors of agent $i$. 

\begin{standing}
	\label{Ass:Graph}
	The communication graph $\mathcal G (\mc{I},\mc{E}) $ is undirected and connected. 
	\hfill $\square$
\end{standing}

For ease of notation, we  also assume that every node of the graph has a self-loop and that $W$ is doubly stochastic; this condition is not strictly necessary and can be dropped (see §\ref{sec:convergence}, Remark~\ref{rem:nondoubly}). 
 Nonetheless, we note that  such a mixing matrix can be distributedly generated on any undirected graph, e.g., by assigning Metropolis weights \cite[§2]{Bianchi_LCSS2020}.

\begin{standing}
	\label{Ass:Graph2}
The mixing matrix $W$ satisfies  the following conditions:
\begin{itemize}[topsep=0em]
		\item[(i)] \emph{Self loops: }$w_{i,i}>0$ for all $i\in \mc{I}$;
		\item[(ii)] \emph{Symmetry: }$W=W^\top$;
		\item[(iii)] \emph{Double stochasticity: }$W\1_N=\1_N, \1^\top W=\1^\top$.	{\hfill $\square$}
\end{itemize} 
\end{standing}


To cope with the lack of knowledge, the usual assumption for the partial-decision information scenario  is that each agent keeps an estimate of all other agents' action \cite{TatarenkoShiNedic2018}, \cite{SalehisadaghianiWeiPavel2019} \cite{DePersisGrammatico2018}.  We denote $\x_{i}=\operatorname{col}((\x_{i,j})_{j\in \mc{I}})\in \R^{n}$,  where $\x_{i,i}:=x_i$ and $\x_{i,j}$ is $i$'s estimate of agent $j$'s action, for all $j\neq i$; let also $\x_{j,-i}:=\col((\x_{j,\ell})_{\ell\in\mc{I}\backslash\{i\}})$. Our proposed fully-distributed \gls{NE} seeking dynamics are summarized in Algorithm \ref{algo:1}, where  $\alpha >0$ is a global constant parameter.

 In steady state, agents should agree on their estimates, i.e., $\x_i=\x_j$ for all $i,j \in \mc{I}$. In fact, the updates of the estimates $\x_{i,-i}^{k}$ resembles a consensus protocol, and can be interpreted as the attempt  of the agents to reach an agreement on the time-varying quantity $x$. In turn, the strategy $x_i$ of each agent is updated based on a proximal best-response, augmented with an extra disagreement penalization term. We remark that the agents evaluate their cost functions in their local estimates, not on the actual collective strategy.


\begin{remark}
The functions $J_i(\cdot,\x_{i,-i})$ are strongly convex, for all $\x_{i,-i}$, for all $i\in \mc{I}$, as a consequence of Standing Assumption~\ref{Ass:StrMon}. Therefore the $\argmin$ operator in Algorithm~\ref{algo:1} is single-valued and the algorithm is always well-defined.
\hfill $\square$
\end{remark}

\section{Convergence analysis}\label{sec:convergence}
In this section, we first derive Algorithm~\ref{algo:1} as a \gls{PPPA}. Then, we prove its convergence by leveraging a \emph{restricted} monotonicity property, under which classical results for the \gls{PPA} of monotone operators still hold.

Before proceeding, 
 we need some definitions. We denote $\x=\col((\x_i)_{i\in\mc{I}})$. Besides, let,   for all $i \in \mc{I}$ \cite[Eq.13-14]{Pavel2018},
\begin{subequations}
	\begin{align}
	\mathcal{R}_{i}:=&\left[ \begin{array}{lll}{{0}_{n_{i} \times n_{<i}}} & {I_{n_{i}}} & {\0_{n_{i} \times n_{>i}}}\end{array}\right], 
	\\
	\mathcal{S}_{i}:=&\left[ \begin{array}{ccc}{I_{n<i}} & {\0_{n<i \times n_{i}}} & {\0_{n<i \times n>i}} \\ {\0_{n>i \times n<i}} & {\0_{n>i \times n_{i}}} & {I_{n>i}}\end{array}\right],
	\end{align}
\end{subequations}
where $n_{<i}:=\sum_{j=1}^{i-1} n_{j}$, $n_{>i}:=\sum_{j=i+1}^{N} n_{j}$. In simple terms, $\mathcal R _i$ selects the $i$-th $n_i$ dimensional component from an $n$-dimensional vector, while $\mathcal S_i$ removes it. Thus, $\mathcal{R}_{i} \x_{i}=\x_{i,i}=x_i$ and $\mathcal{S}_{i} \x_{i}=\x_{i,-i}$. Let $\mathcal{R}:=\operatorname{diag}\left((\mathcal{R}_{i})_{i \in \mathcal{I}}\right)$,  $\mathcal{S}:=\operatorname{diag}\left((\mathcal{S}_{i})_{i \in \mathcal{I}}\right)$. Hence, $\mathcal{R} \x=x$ and $\mathcal{S}{\x}= \operatorname{col}((\x_{i,-i})_{i \in \mathcal{I}})\in \R^{(N-1) n}$. Moreover,
$
\x=\mathcal{R}^\top x+\mathcal{S}^\top \mc{S} \x.
$
We define the  \emph{extended pseudo-gradient} operator  $\bs{F}$ as
\begin{align}
\label{eq:extended_pseudo-gradient}
\bs{F}(\x):=\operatorname{col}\left((\nabla_{x_{i}} J_{i}\left(x_{i}, \x_{i,-i}\right))_{i \in \mathcal{I}}\right),
\end{align}
and the mappings
\begin{align}
\label{eq:Fa}
\bs{F}_{\textnormal{a}}(\x)&:=\alpha\mc{R}^\top\bs{F}(\x)+(I_{Nn}-\bs{W})\x,
\\
\label{eq:opA}
\mathcal{A}(\x)&:=\bs{F}_{\textnormal{a}}(\x)+\mathrm{N}_{\bs{\Omega}}(\x),
\end{align}
where $\alpha>0$ is a fixed parameter, $\bs{W}:=W\otimes I_n$ and 
\begin{align}
\label{eq:Omegabf}
\bs{\Omega}&:=\{\x\in\R^{nN}\mid \mc{R}\x\in \Omega\}.
\end{align}

The following lemma relates the \gls{NE} of the game in \eqref{eq:game} to the operators $\bs{F}_{\textnormal{a}}$ and $\mc{A}$. The proof is analogous to {\cite[Prop.~1]{TatarenkoShiNedic2018}}, and hence it is omitted.

\begin{lemma}\label{lem:VIequivalence}
	The following statements are equivalent:
	\begin{itemize}[topsep=0em] 
		\item [i)] $\bs{x^*}=\1_N\otimes x^*$, with  $x^*\in\Omega$  the \gls{NE} of the game in \eqref{eq:game};
		\item[ii)] $\x^*$ solves VI$(\bs{F}_\textnormal{a},\bs{\Omega})$;
		\item[iii)] $\0_{Nn}\in \A(\x^*)$.  \hfill $\square$
	\end{itemize}
\end{lemma}

\smallskip
\subsection{Derivation of the algorithm}

Lemma \ref{lem:VIequivalence} is fundamental, because it  provides a systematic way of deriving fully-distributed \gls{NE} seeking algorithms, by applying standard solution methods for  VI$(\bs{F}_\textnormal{a},\bs{\Omega})$ (e.g., in \cite{TatarenkoShiNedic2018}, a projected gradient-method was developed) or operator splitting methods to compute a zero of the operator $\A(\x)$ (e.g.,  \cite{Pavel2018} follows a similar approach for games with coupling constraints). Nonetheless, technical difficulties arise because of the partial-decision information assumption. Specifically, the operator $\mc{R}^\top \bs{F}$ is not monotone for most cases of interest, not even if strong monotonicity of the pseudo-gradient mapping $F$ holds, i.e., Standing Assumption \ref{Ass:StrMon}. Only when the estimates $\bs{x}$ belong to the consensus subspace, i.e. $\bs{x}=\1_N\otimes x$, we have that $\bs{F}(\bs{x})=F(x)$.

In fact, Algorithm~\ref{algo:1} is an instance of (suitably preconditioned) \gls{PPA} \cite[Th.~23.41]{Bauschke2017} to seek a zero of $\A$. 
 We remark that many operator-theoretic properties are not guaranteed for the resolvent of a non-monotone operator $\mc{B}:\R^m\rightrightarrows \R^m$. For example,  $\J_{\mc{B}}=(\Id+\mc{B})^{-1}$
 may have a limited domain, or be set-valued. 
In this general case,
we write the \gls{PPA} as
\begin{equation}\label{eq:PPPiter}
\omega^{k+1}\in \J_\mc{B}(\omega^{k}),
\end{equation}
that is  well defined only if $\J_\mc{B}(\omega^k)\neq \varnothing$. Next, we show that Algorithm~\ref{algo:1} is obtained by applying the iteration in  \eqref{eq:PPPiter} to the operator
 $\Phi^{-1}\A$, where 
\begin{equation}\label{eq:Phi}
\Phi:= I_{Nn}+\bs{W},
\end{equation} 
is a symmetric nonsingular matrix, known as \emph{preconditioning} matrix. We note that $\Phi\succ 0$, by Standing Assumption~\ref{Ass:Graph2} and Gershgorin circle theorem,  and that  $\zer(\A)=\zer(\Phi^{-1}\A)$.

\begin{lemma}\label{lem:derivation}
Algorithm~\ref{algo:1} is equivalent to
	\begin{equation}
	\label{algo:1compact}
	\x^{k+1}\in \J_{\Phi^{-1}\A}(\x^k),
		\end{equation}
		 with $\A$ as in \eqref{eq:opA}, $\Phi$ as in \eqref{eq:Phi}.
	\hfill $\square$
\end{lemma}

\begin{proof} 
	By definition of inverse operator we have that
	\allowdisplaybreaks
	\begin{align}
	\nonumber
	\hspace{-1.5 em}\x^{k+1} \! & \in (\Id+\Phi^{-1}\A)^{-1}\x^{k}
	\\
	\nonumber
	\hspace{-1.5 em}\iff 
	\0_{Nn}  & \in \x^{k+1}+\Phi^{-1} \A\x^{k+1}-\x^{k}
	\\
	\nonumber
	\hspace{-1.5 em}\iff 
	\0_{Nn}   & \in \Phi(\x^{k+1}-\x^{k})+\A\x^{k+1} 
	\\
\nonumber
	\hspace{-1.5 em}\iff 
	\0_{Nn} & \in \x^{k+1}+\cancel{\bs{W}\x^{k+1}}-\x^{k}-\bs{W}\x^{k} 	+\x^{k+1}
	\\ 
	\label{eq:stepderivation}
	&
    \quad -\cancel{\bs{W}\x^{k+1}}+\alpha\mc{R}^\top \bs{F}(\x^{k+1})+ \mathrm{N}_{\bs{\Omega}}(\x^{k+1}).  \hspace{-1 em}
	\end{align}
	In turn, the last inclusion can be split in two components by left-multiplying both sides with $\mc{R}$ and $\mc{S}$. Hence, by  $\mc{S}\mathrm{N}_{\bs{\Omega}}=\0_{(N-1)n}$ and  $\mc{S}\mc{R}^\top=\0_{(N-1)n\times n}$,  \eqref{eq:stepderivation} is equivalent to
	\begin{align}
	\nonumber
    &
	\left\{ \begin{aligned}
	 \0_{(N-1)n} & \in   \mc{S}(2\x^{k+1}-\x^{k})-\mc{S}\bs{W}\x^{k}
	\\  \0_{n} &\in
	\begin{multlined}[t]
	  2x^{k+1}-x^k-\mc{R}\bs{W}\x^k+\mathrm{N}_{{\Omega}}(x^{k+1})\\
	+ \alpha\bs{F}(x^{k+1},\mc{S}\x^{k+1})
	\end{multlined}
	\end{aligned} 
	\right.
	\\
	\nonumber
	\underset{\forall i\in \mc{I}}{\iff} &
	\left\{ \begin{aligned}
	\x_{i,-i}^{k+1} &= \textstyle \frac{1}{2}(\x_{i,-i}^k+\textstyle\sum_{j=1 }^{N}w_{i,j}\x_{j,-i}^{k})
	\\
	\nonumber
	\0_{n_i} &\in  \partial_{x_i^{k+1}}  \bigl(
	\begin{aligned}[t]
	& J_i(x_i^{k+1},\x_{i,-i}^{k+1}) +\textstyle \frac{1}{2\alpha}\bigl\| x_i^{k+1}-x_i^{k} \bigr\|^2 \\
	&+ \textstyle\frac{1}{2\alpha}\bigl\| x_i^{k+1}- \textstyle \sum_{j=1}^{N}w_{i,j}\x_{j,i}^k\bigr\|^2 
	\\ &
	+\iota_{\Omega_i}(x_i^{k+1})\bigr).
	\end{aligned}
	\end{aligned}
	\right.
	\end{align}
	The conclusion follows since the zeros of the  subdifferential of a (strongly) convex function coincide with the minima (unique minimum) \cite[Th.~16.3]{Bauschke2017}.
\end{proof}

\begin{remark}
The preconditioning matrix $\Phi$ is designed to decouple the system of inclusion in \eqref{eq:stepderivation}
from the graph structure,
i.e., to remove the term $\bs{W} \bs x^{k+1}$.
This ensures that the resulting updates can be computed by the agents  in a fully-distributed fashion.
 \hfill $\square$
\end{remark}

\begin{remark}\label{rem:rem1}
	By Lemma~\ref{lem:derivation}  and the explicit expression of  $\J_{\Phi^{-1}\A}$ in Algorithm~\ref{algo:1}, we conclude that  $\dom(\J_{\Phi^{-1}\A})=\R^{Nn}$ and that $\J_{\Phi^{-1}\A}$ is single-valued on its domain. \hfill $\square$
\end{remark}

\subsection{Convergence analysis}
Since the operator $\Phi^{-1}\A$ is not maximally monotone in general, the  convergence of Algorithm~\ref{algo:1} cannot be inferred by standard results for the \gls{PPA}.
In \cite[Th.~5]{TatarenkoShiNedic2018}, the authors proposed an accelerated gradient \gls{NE} seeking scheme, which achieves geometric convergence if the mapping $\bs{F }_{\textnormal{a}}$ is strongly monotone. However, this is a limiting assumption, that can be guaranteed only for some classes of games (cf. \cite[Rem.~3]{TatarenkoShiNedic2018}). Instead, our analysis is based on a weaker condition, namely the \emph{restricted }strong monotonicity of $\bs{F }_{\textnormal{a}}$ only with respect to the \gls{NE} \cite{Pavel2018}, \cite{SalehisadaghianiWeiPavel2019}. The main advantage is that, for suitable choices of the parameter $\alpha$, restricted strong monotonicity (unlike strong monotonicity) of $\bs{F}_{\textnormal{a}}$ holds for any game satisfying Standing Assumptions~\ref{Ass:Convexity}-\ref{Ass:Graph2}, as formalized in the next two  statements.

\begin{lemma}[\!\!{\cite[Lem.~1]{Bianchi_LCSS2020}}]\label{lem:LipschitzExtPseudo}
	The  mapping $\bs{F}$ in \eqref{eq:extended_pseudo-gradient} is $\theta$-Lipschitz continuous, for some $ \theta\in[\mu,\theta_0]$: for any  $\x,\y\in \R^{Nn}$, $\| \bs{F} (\x)-\bs{F}(\y)\|\leq \theta \|\x-\y\|$.
	{\hfill $\square$} \end{lemma}


\begin{lemma}[\!\!{\cite[Lem.~$3$]{Pavel2018}}]\label{lem:strongmon_constant}
	Let 
	\begin{align}
	\label{eq:M1}
    {M}:=\alpha\begin{bmatrix}{\frac{\mu}{N}} & \ {-\frac{\theta_0+\theta}{2\sqrt{N}}} \\ {-\frac{\theta_0+\theta}{2\sqrt{N}}} & \ { \textstyle \frac{\uplambda_{2}(I-W)}{\alpha}-\theta}\end{bmatrix}, 
	\quad \begin{aligned}
	\alphamax&:=\textstyle \frac{4\mu\uplambda_2({I-W})}{(\theta_0+\theta)^{2}+4\mu\theta},
	\\
	{\rho}_\alpha&:=\uplambda_{\textnormal{min}} ({M}).
	\end{aligned}
	\end{align}
	For any $\alpha \in (0,\alphamax)$, $\bs{F}_{\textnormal{a}}$ is $\rho_\alpha$-restricted strongly monotone with respect to the consensus subspace $\bs{E}_{n} :=\{\bs{y} \in \R^{N n}:\bs{y}=\1_{N} \otimes y, y\in \R^n\}$: for any $\x\in\R^{Nn}$ and any $\y\in \bs{E}_{n}$,  it holds that $M\succ 0$  and also that 
	\[
	\begin{multlined}
	\langle\x-\y, 
	\bs{F}_{\textnormal{a}} (\x)-\bs{F}_{\textnormal{a}}  (\y)\rangle \geq {\rho_\alpha}\left\|\x-\y \right\|^{2}.
 	\end{multlined}
	\QEDopenhereeqn
	\]
\end{lemma}


\smallskip
Moreover, the operator $\Phi^{-1}\A$ retains this property, in the Hilbert space induced by the inner product $\langle \cdot , \cdot \rangle_{\Phi}$. Here, we denote, for a matrix $\Psi\succ 0$,  by $\langle\cdot ,\cdot\rangle_\Psi : (x,y)\mapsto  \langle\Psi x, y\rangle$ and  $\| \cdot \|_\Psi : x\mapsto  \langle\Psi x , x\rangle$ the $\Psi$-weighted Euclidean inner product and norm, respectively.

\begin{lemma}\label{lem:strmoninPhi}
	Let $\alphamax$, ${\rho_\alpha}$ be as in \eqref{eq:M1},   $\alpha\in(0,\alphamax)$. Then, $\Phi^{-1} \A$ is $\textstyle \frac{{\rho_\alpha}}{\|\Phi \|}$-restricted strongly monotone with respect to $\zer(\A)$ in the $\Phi$-induced space: for all $(\x,\bs{u}),(\y,\bs{v})\in \gra({\Phi^{-1}\mc{A}})$ such that $\bs{y} \in \zer(\mc{A})$, it holds that 
	\[\langle \bs{u}-\bs{v} , \x-\y \rangle_{\Phi} \geq \textstyle \frac{{\rho_\alpha}}{\|\Phi \|} \|\x-\y\|^2_{\Phi}. \QEDopenhereeqn \]
\end{lemma} 
\smallskip

Towards our main result, we next prove the convergence of the iteration in \eqref{eq:PPPiter} to an equilibrium, under restricted strong monotonicity of  $\mc{B}$. The proof is based on the restricted contractivity of $\J_\mc{B}$ with respect to its (unique) fixed point.

\begin{theorem}\label{th:PPP}
	Let $\mc{B}:\R^m\rightrightarrows \R^m$ be $\tilde{\mu}$-restricted strongly monotone with respect to $\zer(\mc{B})$, for some $\tilde{\mu}>0$, in the space induced by some inner product $\langle \cdot,\cdot \rangle_{\Psi}$, i.e.,
	\begin{equation}\label{eq:thPPPrestricted}
	\langle x-z , u-v \rangle_{\Psi} \geq \tilde{\mu} \|x-z\|^2_{\Psi},
	\end{equation}
	for any $(x,u),(z,v)\in \gra(B)$ such that $z\in \zer(\mc{B})$. Assume that $\zer(\mc{B})\neq \varnothing$ and that $\dom(\J_\mc{B})=\R^m$. Then, for any $\omega^0\in\R^m$,  the proximal-point iteration in \eqref{eq:PPPiter}
	converges with linear rate to the unique point $\{z\}= \zer(\mc{B})$: for all $k\in\N$,
\[
	\textstyle	\|\omega^k-z\|_{\Psi}\leq\left(\frac{1}{ 1+ \tilde{\mu}}\right)^{k}\|\omega^0-z \|_{\Psi}. \QEDopenhereeqn 
\]
\end{theorem}
\smallskip

We are now ready to show the main result of the paper, namely the linear convergence of Algorithm~\ref{algo:1}  to a \gls{NE} .

\begin{theorem}\label{th:main}
	Let $\alphamax$, ${\rho_\alpha}$  be as in \eqref{eq:M1},  $\Phi$ as in \eqref{eq:Phi}, $\alpha\in (0,\alphamax)$. For any initial condition, the sequence $(\x^k)_{k\in\N}$ generated by Algorithm~\ref{algo:1} converges with linear rate to  $\x^*=\1_n\otimes x^*$, where $x^*$ is  the unique \gls{NE} of the game in \eqref{eq:game}: for all $k\in \N$,   
	\[ 
	\textstyle \|\x^k-\x^*\|\leq \sqrt{ \frac{\| \Phi \|}{\uplambda_{\textnormal{min}}(\Phi)}}\left(\frac{1}{1+{\rho_\alpha}/\|\Phi\|}\right)^k\|\x^0-\x^*\|. \QEDopenhereeqn
	\]
\end{theorem}
\smallskip

\begin{remark}\label{rem:nondoubly}
	If the mixing matrix $W$ is not doubly stochastic (i.e., if the last condition in Standing Assumption~\ref{Ass:Graph2} does not hold), an iteration analogous to Algorithm~\ref{algo:1} can be derived by defining  $\mc{A}(\x)=\alpha\mc{R}^\top\bs{F}(\x)+(\bs{D}-\bs{W})\x+\mathrm{N}_{\bs{\Omega}}(\x)$ and $\Phi=\bs{D}+\bs{W}$, where $\bs{D}= D\otimes I_n$ and $D$ is the degree matrix of the communication graph, for which  a convergence result analogous to Theorem~\ref{th:main} still holds.  \hfill $\square$
\end{remark}

\section{Discussion on the convergence rate}\label{sec:theoreticalrates}
In this section, we compare the convergence rate of Algorithm~\ref{algo:1} with that of two gradient-based \gls{NE} seeking schemes recently presented in \cite{TatarenkoShiNedic2018}: GRANE \cite[Alg.~1]{TatarenkoShiNedic2018} and  acc-GRANE \cite[Alg.~2]{TatarenkoShiNedic2018}.

GRANE converges linearly  under $\mu_{\bs{F}_{\textnormal{a}}}$-restricted strongly monotonicity and $\theta_{\bs{F}_{\textnormal{a}}}$-Lispchitz continuity of $\bs{F}_\textnormal{a}$ in \eqref{eq:Fa}, with rate (in squared norm) $\mc{O}(\lambda^k_{\textnormal{\tiny GRANE}})$, where
\begin{equation}\tag{GRANE}
\textstyle \lambda_{\textnormal{\tiny GRANE}}=   1-\frac{1}{\theta_{\bs{F}_{\textnormal{a}}}^2/\mu_{\bs{F}_{\textnormal{a}}}^2}.~~\end{equation}
Instead, acc-GRANE is only  guaranteed to converge under the more restrictive assumption of (non-restricted) $\bar{\mu}_{\bs{F}_{\textnormal{a}}}$-strong monotonicity of 
the mapping $\bs{F}_{\textnormal{a}}$ (which requires some additional conditions on the game mapping, see \cite[Rem.~3]{TatarenkoShiNedic2018}), where $\bar{\mu}_{\bs{F}_{\textnormal{a}}}\leq \mu_{\bs{F}_{\textnormal{a}}}$,  with squared norm rate $\mc{O}(\lambda^k_{\textnormal{\tiny acc-GRANE}}) $, where 
\begin{equation} \tag{acc-GRANE}
	\textstyle ~~~~~~~~~~~~ \lambda_{\textnormal{\tiny acc-GRANE}}=1-\frac{1}{1+\theta_{\bs{F}_{\textnormal{a}}}/\bar{\mu}_{\bs{F}_{\textnormal{a}}}}.
\end{equation}
\allowbreak
 Finally, the convergence rate of  Algorithm~\ref{algo:1}, in squared \linebreak \pagebreak 
 \newline
 norm, is $\mc{O}(\lambda_{\textnormal{\tiny PPP}}^k) $, where
 \begin{align}
 \tag{PPP}
 \textstyle  \lambda_{\textnormal{\tiny PPP}} = \left( \frac{1}{1+{\mu_{\bs{F}_{\textnormal{a}}}}/\|\Phi\|}\right)^2.~~~~~~~
 \end{align}
 Since $\|\Phi\|=2$ due to Standing Assumption \ref{Ass:Graph2}, we get
\begin{subequations}
\begin{align}
 \lambda_{\textnormal{\tiny PPP}} 
\label{eq:rate1}
& = 1+ 1/(1+2/\mu_{\bs{F}_{\textnormal{a}}})^2-2/(1+2/\mu_{\bs{F}_{\textnormal{a}}})
  \\ 
  \label{eq:rate2}
& \leq   1-1/(1+2/\mu_{\bs{F}_{\textnormal{a}}}),
\end{align}
\end{subequations}
We note that $\uplambda_{2}(I-W)\leq2$, and that $\uplambda_{\textnormal{max}}(I-W)\geq 1$, if the self-loop weights are chosen small enough.
By the expression of $\alphamax$ in \eqref{eq:M1} and by picking $\alpha < \alphamax$ to ensure restricted strong monotonicity, it can be shown that $\alpha \bs{F}$ is $\frac{\uplambda_{2}(I-W)}{2}$-Lipschitz continuous, 
and in turn that
$3\geq \uplambda_{\textnormal{max}}(L)+\uplambda_{2}(L)/2 \geq \theta_{\bs{F}_{\textnormal{a}}}\geq \uplambda_{\textnormal{max}}(L)-\uplambda_{2}(L)/2 \geq 1/2$. 
%
%

We conclude that, when
$\mu_{\bs{F}_{\textnormal{a}}}$ is small (which is typically the case), Algorithm~\ref{algo:1} has a faster theoretical convergence rate than GRANE, by \eqref{eq:rate2}. Moreover, if $\theta_{\bs{F}_{\textnormal{a}}}>1$, then the guaranteed rate of Algorithm~\ref{algo:1} is better than that of  acc-GRANE by \eqref{eq:rate1}, despite our \gls{PPPA} is ensured to converge under milder conditions and requires only one communication per iteration instead of two.
%
%
%
%
%
%

\section{Numerical example: a connectivity problem}\label{sec:numerics}
 \setlength{\textfloatsep}{\textfloatsep-1em}
\begin{figure}[t]
	\centering
	\includegraphics[width=\columnwidth]{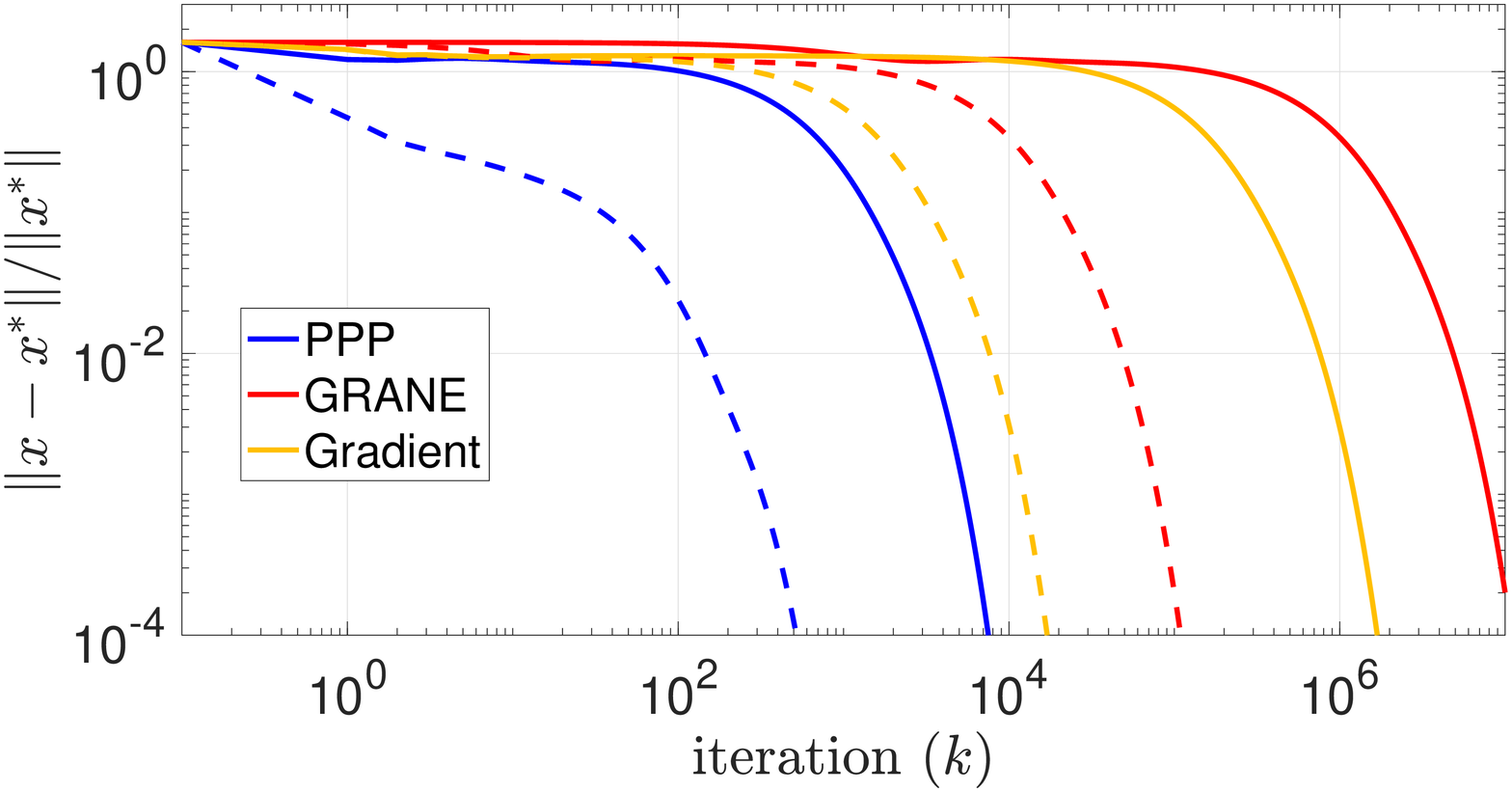}
	\caption{Comparison of \gls{PPPA}, gradient play and GRANE, with the theoretical step sizes that ensure convergence (solid lines) and with step sizes set 100 times larger than the theoretical upper bounds (dashed lines).}
	\label{fig:1}
\end{figure} 
%
%
We consider the connectivity problem described in \cite{Stankovic_Johansson_Stipanovic_2012}. Some mobile sensor devices have to coordinate their
actions via wireless communication, to perform some task, e.g., exploration or surveillance. Mathematically,  the sensors (agents) aim at autonomously  finding the positions which minimize some global cost. This can be robustly achieved by designing individual
cost functions for the agents,  such that the Nash equilibrium of the resulting game coincide with 
an optimum of the global objective  \cite{Stankovic_Johansson_Stipanovic_2012}.
Specifically, each agent $i$ of a group $\mc{I}=\{1,\dots,N \}$ is a mobile sensor moving on a plane,
designed to achieve some private primary objective related
to its Cartesian position $x_i\in \R^2$, provided that overall connectivity is preserved over the network. This is represented by the cost functions
$J_i(x_i,x_{-i})=q_i x_i^\top x_i +r_i x_i+\textstyle \sum_{j=1}^{N}m_{i,j}\|x_i-x_j\|^2,$
where $q_i>0$, $r_i$ and $m_{i,j}\geq 0$ are local parameters, for all $i,j\in\mc{I}$.
 The agents cannot measure the positions of the other sensors, but  communicate with some neighbors over a (randomly generated) communication network.
We set $N=10$, $m_{i,j}=1/N$ $\forall i,j\in\mc{I}$; we pick randomly with uniform distribution $q_i$ in $[1,2]$ and $r_i$ in $[-2,2]$. Because of the quadratic structure of the game and the choice of the parameters, all of our assumptions are satisfied. We compare the performance of Algorithm~\ref{algo:1} with that of some gradient-based \gls{NE} seeking algorithms proposed in the literature, for random initial conditions.
%

\emph{Unconstrained action sets}: 
We compare Algorithm~\ref{algo:1} with GRANE  \cite[Alg.~1]{TatarenkoShiNedic2018} and the gradient play in \cite[Eq.~7]{TatarenkoNedic2019_unconstrained}. We set $\alpha\approx 10^{-2}$, which satisfies the condition in Theorem~\ref{th:PPP}; for the other  algorithms we choose the best step sizes with theoretical convergence guarantees. Figure~\ref{fig:1} (solid lines) shows that
both gradient algorithms  are outperformed by far by our \gls{PPPA}. Indeed, both GRANE and the gradient play  are converging very slowly to the \gls{NE}, mostly due to the small step sizes employed.
 However, our numerical experience suggests that the theoretical bounds for the parameters are very conservative. Hence, we repeat the experiment by taking all the step sizes $100$ times bigger than their theoretical upper bounds (dashed lines). The convergence appears faster for all the algorithms, but Algorithm~\ref{algo:1} is still orders of magnitude better than the gradient-based schemes.
 
\emph{Constrained action sets}: 
We assume each coordinate of the position of each sensor to be constrained in the interval  $[0.1, 0.5 ]$. We test  Algorithm~\ref{algo:1} against GRANE \cite[Alg.~1]{TatarenkoShiNedic2018},  the inexact ADMM algorithm in \cite[Alg.~1]{SalehisadaghianiWeiPavel2019}, and 
acc-GRANE  \cite[Alg.~2]{TatarenkoShiNedic2018} (the latter is guaranteed to converge only under \emph{non-restricted} strong monotonicity of the mapping $\bs{F}_{\textnormal{a}}$ in \eqref{eq:Fa}: we check numerically that this condition holds for our case study).
For all the algorithms, we select  the best  step sizes with theoretical guarantees. 
The results are illustrated in Figure~\ref{fig:3}. 
\begin{figure}[t]
	\centering
	\includegraphics[width=\columnwidth]{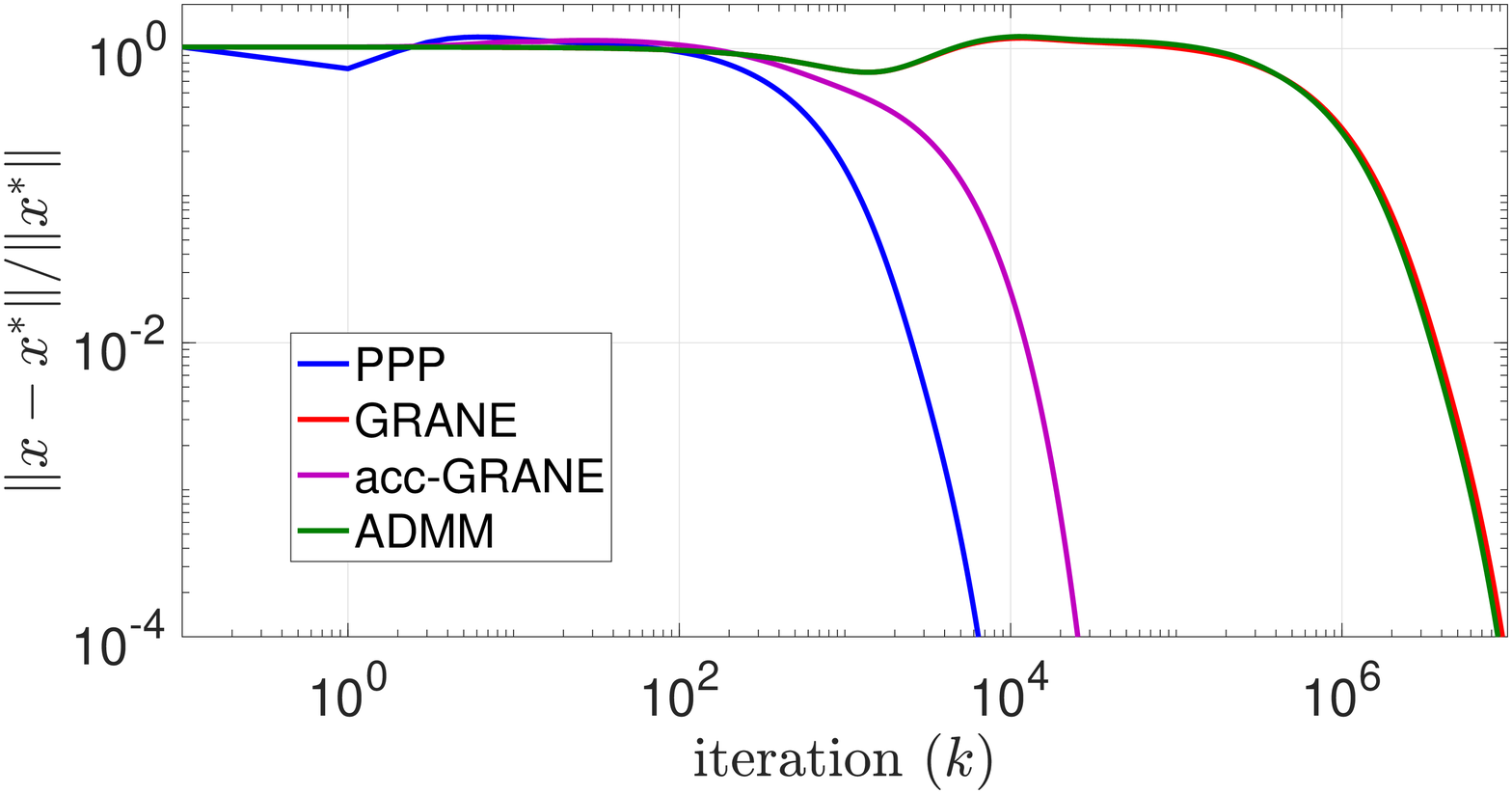}
	\caption{Comparison of \gls{PPPA}, GRANE, and acc-GRANE, with theoretical step sizes and compact action sets.}
	\label{fig:3}
\end{figure}

\begin{remark}
		{Our PPPA requires each agent to solve a strongly convex optimization problem at each iteration. While this can be efficiently done via iterative algorithms, it might be more demanding than a projected pseudo-gradient step, which in general requires to solve a strongly convex \emph{quadratic} optimization problem. Nonetheless, our simulations show that our method can drastically reduce the number of iterations, and thus of data transmissions, needed to converge. This is often advantageous, even at the price of an increased local computational effort, since communication between agents is typically expensive in terms of both time and energy consumption (especially in wireless systems).} \hfill $\square$
\end{remark}

\section{Conclusion and outlook}\label{sec:conclusion}
Nash equilibrium problems under partial-decision information can be solved via a fully-distributed preconditioned proximal-point algorithm, under strong monotonicity and Lipschitz continuity of the game mapping. Our algorithm has proven much faster than the existing gradient-based methods, at least in our numerical experience. The extension of our results to games with coupling constraints or played on time-varying communication networks is left as future research.

\appendix

\subsubsection{Proof of Lemma~\ref{lem:strmoninPhi}}\label{app:lem:strmoninPhi}
Let $(\x,{\bs{u}}), (\y,{\bs{v}})\in\gra({\Phi^{-1}\mc{A}})$ such that $\bs{y}\in\zer(\mc{A})=\zer(\Phi^{-1}\mc{A})$. By definition, $(\x,\Phi{\bs{u}}),(\y,\Phi{\bs{v}})\in\gra(\mc{A})$. Besides, the operator $\A$ in \eqref{eq:opA} is ${\rho_\alpha}$-restricted strongly monotone with respect to the consensus space $\bs{E}_{n}$, by Lemma~\ref{lem:strongmon_constant} and monotonicity of the normal cone \cite[Th.~20.25]{Bauschke2017}. Since $\bs{y} \in  \bs{E}_{n}$ by Lemma~\ref{lem:VIequivalence}, we can write:
\[
\begin{aligned}[b]
 \langle \bs{u}-\bs{v} , \x-\y \rangle_{\Phi} &= \langle \Phi\bs{u}-\Phi\bs{v} , \x-\y \rangle
\\
& \geq  {\rho_\alpha} \|\x-\y\|^2 \geq \textstyle \frac{ {\rho_\alpha} }{\|\Phi\|}\|\x-\y\|_\Phi^2.
\end{aligned} \QEDhereeqn
\]

\subsubsection{Proof of Theorem~\ref{th:PPP}}\label{app:th:PPP}
We start by noting that $\fix(\J_B)=\zer(B)$, since, for any $z\in\R^m$, $\0 _m\in \mc{B}(z) \iff z\in z+ \mc{B}(z) \iff (\Id+\mc{B})^{-1} (z)\ni z$.
Let $x\in\R^m$, $y\in\J_{\mc{B}}(x)=(\Id+\mc{B})^{-1}(x)$, $z\in\zer({\mc{B}})$. 
By definition of inverse operator,  $x-y \in \mc{B}(y)$;  therefore, by taking $v=\0_m$ in  \eqref{eq:thPPPrestricted},  we have
\begin{equation}
\begin{aligned}
&\langle y-z  , x-y-(z-z) \rangle_{\Psi}
\\
=& -\|y-z\|^2_{\Psi} +\langle y-z , x-z \rangle_{\Psi}\geq \tilde{\mu}\|y-z\|^2_{\Psi}.
\end{aligned}
\end{equation}
In turn, by the Cauchy-Schwarz inequality we obtain 
\begin{equation}\label{eq:usefulstep1}
	\|y-z\|_{\Psi} \|x-z\|_{\Psi} \geq (1+\tilde{\mu})\|y-z\|^2_{\Psi}.
	\end{equation}
Let us set $x=w^k$ and $y=w^{k+1}$. If $y\neq z$, by dividing both sides of \eqref{eq:usefulstep1} by $\|y-z\|_{\Psi}$, we obtain for the iteration in  \eqref{eq:PPPiter}:
\[
	\|\omega^{k+1}-z\|_{\Psi}\leq \left( \textstyle \frac{1}{1+\tilde{\mu}}\right)\|\omega^k-z\|_{\Psi};
\]
this trivially holds also if $y=z$.
By recursion, we conclude linear convergence of $(\omega^k)_{k\in \N}$ to $z$. As $z$ is an arbitrary point in $\zer(\mc{B})$, it also follows that $\zer(\mc{B})$ is a singleton. {\hfill $\blacksquare$}
\smallskip


\subsubsection{Proof of Theorem~\ref{th:main}}\label{app:th:main}
 Algorithm~\ref{algo:1} can be written as the  \gls{PPA} in \eqref{algo:1compact} by Lemma~\ref{lem:derivation}. Since $\zer(\A)\!=\! \zer(\Phi^{-1}  \A) \!=\! \{\x^*\}$ by Lemma~\ref{lem:VIequivalence},  $\dom(\J_{\Phi^{-1}\A})=\R^{Nn}$  by Remark~\ref{rem:rem1} and in view of  Lemma~\ref{lem:strmoninPhi}, we  can apply Theorem~\ref{th:PPP} to  conclude that
\[ \begin{aligned}[b]
\textstyle \sqrt{\uplambda_{\textnormal{min}}(\Phi)} \|\x^k-\x^*\| & \leq  \|\x^k-\x^*\|_{\Phi} 
\\ 
& 
 \leq  \left(\textstyle \frac{1}{1+{\rho_\alpha}/\|\Phi\|}\right)^{\!k}\|\x^0\! -\x^*\|_{\Phi} 
 \\ & \leq  \sqrt{\|\Phi \|}\left(\textstyle \frac{1}{1+{\rho_\alpha}/\|\Phi\|}\right)^{\!k}\|\x^0 \! -\x^*\|.
 \! 
\end{aligned}\QEDhereeqn \]


\bibliographystyle{IEEEtran}
\bibliography{library}
\end{document}